\documentclass{amsart}
\usepackage[utf8]{inputenc}
\usepackage{amsmath}
\usepackage{hyperref}
\usepackage{tikz}
\usetikzlibrary{cd}
\usetikzlibrary{arrows}
\usetikzlibrary{decorations.pathreplacing}
\newtheorem {theorem}{Theorem}

\newtheorem {corollary}[theorem]{Corollary}

\newcommand{\Kh}{\mathit{Kh}}
\newcommand{\Khr}{\overline{\mathit{Kh}}}
\newcommand{\HFK}{\widehat{\mathit{HFK}}}
\newcommand{\CFK}{\mathit{CFK}^\infty}
\newcommand{\CF}{\widehat{\mathit{CF}}}
\newcommand{\HF}{\widehat{\mathit{HF}}}
\newcommand{\Q}{\mathbb{Q}}
\newcommand{\Z}{\mathbb{Z}}

\title{Khovanov homology detects the figure-eight knot}
\date{}

\author[J.A. Baldwin]{John A. Baldwin}
\thanks{JAB was partially supported by NSF CAREER grant DMS-1454865}
\address{Department of Mathematics, Boston College}
\email{john.baldwin@bc.edu}

\author[N. Dowlin]{Nathan Dowlin}
\thanks{ND was partially supported by NSF grant DMS-1606421}
\address{Department of Mathematics, Columbia University}
\email{ndowlin@math.columbia.edu}

\author[A.S. Levine]{Adam Simon Levine}
\thanks{ASL was partially supported by NSF grant DMS-1806437.}
\address{Department of Mathematics, Duke University}
\email{alevine@math.duke.edu}

\author[T. Lidman]{Tye Lidman}
\thanks{TL was partially supported by a Sloan Fellowship and NSF grant DMS-1709702.}
\address{Department of Mathematics, North Carolina State University}
\email{tlid@math.ncsu.edu}

\author[R. Sazdanovic]{Radmila Sazdanovic}
\thanks{RS was partially supported by NSF grant DMS-1854705.}
\address{Department of Mathematics, North Carolina State University}
\email{rsazdanovic@math.ncsu.edu}

\begin{document}
\maketitle

\begin{abstract}
Using Dowlin's spectral sequence from Khovanov homology to knot Floer homology, we prove that reduced Khovanov homology (over $\mathbb{Q}$) detects the figure-eight knot.
\end{abstract}

Khovanov homology is a powerful invariant of knots and links, and yet its precise connections with topology remain largely mysterious. One of the most basic topological questions one can ask about a knot invariant is which knots it detects. In 2010, Kronheimer and Mrowka used instanton gauge theory to prove that Khovanov homology detects the unknot \cite{KMunknot}. In 2018, Baldwin and Sivek combined gauge theory with ideas from contact topology to prove that Khovanov homology detects the trefoils \cite{BStrefoil}. Until now, these were the only \emph{knots} known to be detected by Khovanov homology, though there have been several additional results (most very recent) regarding detection of links whose components are unknots and trefoils; see Hedden--Ni \cite{hedden-ni}, Batson--Seed \cite{batson-seed}, Baldwin--Sivek--Xie \cite{BSXhopf}, Xie--Zhang \cite{XZunlink,XZsmall}, Lipshitz--Sarkar \cite{LipshitzSarkar}, Martin \cite{Martin}, Li--Xie--Zhang \cite{LXZsmall}.

In this note, we use Dowlin's  spectral sequence from reduced Khovanov homology to knot Floer homology  \cite{Dowlin}  to prove that the former also detects the figure-eight knot $4_1$. Note that the reduced Khovanov homology of $4_1$ (over $\Q$) is given by \[\Khr(4_1;\Q) \cong \Q_{(-4,-2)}\oplus \Q_{(-2,-1)}\oplus \Q_{(0,0)}\oplus \Q_{(2,1)}\oplus \Q_{(4,2)},\] where the subscript $(q,h)$ indicates quantum grading $q$ and homological grading $h$.

Our main theorem is the following. In our conventions,  the $\delta$-grading in reduced Khovanov homology is given by $\delta = q/2-h$.





\begin{theorem}\label{thm:main}
Let $K \subset S^3$ be a knot whose  reduced Khovanov homology over $\Q$ is $5$-dimensional and is supported in a single $\delta$-grading $d$. Then:
\begin{enumerate}
\item If $d=0$, then $K$ is the figure-eight knot.
\item If $d \ne 0$, then, up to mirroring, $d=2$ and $K$ is a genus-2, fibered, strongly quasipositive knot whose bigraded knot Floer homology  over $\mathbb{Q}$ is isomorphic to that of the torus knot $T(2,5)$.
\end{enumerate}
\end{theorem}

Before proving Theorem~\ref{thm:main}, we remark that it is an open question whether knot Floer homology detects $T(2,5)$. Observe the following corollary:
\begin{corollary}
If knot Floer homology over $\Q$ detects $T(2,5)$, then so does reduced Khovanov homology over $\Q$.
\end{corollary}

The rest of this note is devoted to the proof of Theorem~\ref{thm:main}.  From this point on, we will  work solely with coefficients in $\Q$ since Dowlin's spectral sequence requires working over a ring in which  $2$ is invertible.

\begin{proof}[Proof of Theorem~\ref{thm:main}] Suppose $K\subset S^3$ is a knot whose reduced Khovanov homology $\Khr(K)$ is 5-dimensional and is supported in a single $\delta$-grading $d$.
We will first show via case analysis that $K$ is either the figure-eight or, up to mirroring, a genus-2, fibered strongly quasipositive knot with the same knot Floer homology as $T(2,5)$. We will then conclude that $d = 0$ or $2$, respectively, in these two cases.

By \cite{Dowlin}, there is a spectral sequence from $\Khr(K)$ to the knot Floer homology $\HFK(-K)$ which respects the relative $\delta$-grading. Here, $-K$ is  the mirror of $K$, and the $\delta$-grading on knot Floer homology is given by $\delta = m-a$, where $m$ and $a$ are the Maslov and Alexander gradings, respectively. It follows that $\HFK(K)$ is also supported in a single $\delta$-grading, and has total dimension 1, 3, or 5.  If the dimension is 1 or 3, then $K$ is the unknot \cite[Theorem 1.2]{OS:genus} or the trefoil \cite[Corollary 8]{HeddenWatson}, respectively. (The trefoil detection result in \cite{HeddenWatson} is  stated  with coefficients in $\Z/2\Z$, but it holds over $\Q$ as well.) In either case,  the dimension of $\Khr(K)$ is not 5, a contradiction.   Therefore, we must have that $\dim \HFK(K) = 5$.

Since
\begin{equation}\label{eqn:euler-char}
\sum_{m,a} (-1)^m \dim \HFK_m(K,a) t^a = \Delta_K(t)
\end{equation}
and $\Delta_K(1) = 1$, we have that $\dim \HFK(K,0)$ is odd.  Note that $\dim \HFK(K,0) \neq 5$, since the knot Floer homology of the unknot is 1-dimensional and knot Floer homology detects the genus
\[
g(K) = \max \{ a \mid \HFK(K,a) \neq 0\}
\]
by \cite[Theorem 1.2]{OS:genus}.

We next consider the case that $\dim \HFK(K,0) = 3$.  From the symmetry of \cite[Equation 3]{OS:hfk}: \[\HFK(K,a) \cong  \HFK(K,-a),\] we see that there is exactly one positive Alexander grading $i$ in which $\HFK(K,i)$ is nontrivial, and this group is necessarily 1-dimensional.  First, suppose that $i=1$.  Then $K$ is a genus one fibered knot \cite{Ghiggini, juhasz}, which means that $K$ is either a trefoil or the figure-eight.  Since the trefoils have 3-dimensional knot Floer homology,  $K$ must be the figure-eight in this case.  We now show that this Alexander grading $i$ cannot be strictly greater than one.  Suppose for a contradiction that $i>1$. Then $K$ is  a fibered knot of genus $g=i > 1$ \cite{Ghiggini, Ni, juhasz}.  It then follows from work of Baldwin and Vela-Vick \cite{BVV}  that $\HFK(K,g-1) \neq 0$ as well. (This result is stated in  \cite{BVV} with coefficients in $\Z/2\Z$ but also holds over $\Q$.)  But this implies that $\dim \HFK(K) > 5$, a contradiction.

It   remains to consider the case that $\dim \HFK(K,0) = 1$.  We would like to show  in this case that, up to mirroring, $K$ is  a genus-2, fibered, strongly quasipositive knot.  There are two subcases: $\dim \HFK(K,g) = 2$ or 1, where $g$ is the genus of $K$.  We consider these in turn below.

First, suppose that  $\dim \HFK(K,g) = 2$.   Since $\HFK(K)$ is supported in a single $\delta$-grading, the  symmetry
\begin{equation}\label{eqn:symmetry}
\HFK_m(K,a) \cong \HFK_{m-2a}(K,-a)
\end{equation}
 \cite[Equation 3]{OS:hfk} implies that for some integer $M$, the bigraded knot Floer homology of $K$ is given by
\[
\HFK(K) = \Q^2_{(M,g)} \oplus \Q_{(M-g,0)} \oplus \Q^2_{(M-2g,-g)},
\] where the subscript $(m,a)$ indicates  Maslov grading $m$ and Alexander grading $a$. It then follows from \eqref{eqn:euler-char} that, up to sign, we have \[\Delta_K(t) = 2t^g \pm 1 + 2t^{-g}.\]  But this contradicts $\Delta_K(1)=1.$ Therefore, we cannot have $\dim \HFK(K,g) = 2$.

Next, suppose that $\dim \HFK(K,g) = 1$. Then $K$ is fibered, and we may assume that $g\geq 2$ or else $\dim\HFK(K) =3$, which would contradict our  conclusion  that $\dim\HFK(K)=5$. Furthermore, we have that $\dim \HFK(K,g-1) \neq 0$ by \cite{BVV}, as above. In fact, since $\dim \HFK(K) = 5$, we must have that $\dim \HFK(K,g-1) = 1$. Taking into account the symmetry \eqref{eqn:symmetry} and the fact that $\HFK(K)$ is supported in a single $\delta$-grading, this implies that the bigraded knot Floer homology of $K$ is  given by
 \begin{equation*}\label{eqn:hfk}
\HFK(K) = \Q_{(M,g)} \oplus \Q_{(M-1,g-1)} \oplus \Q_{(M-g,0)} \oplus \Q_{(M-2g+1,1-g)} \oplus \Q_{(M-2g,-g)}.
\end{equation*} Let $x_1,\ldots,x_5$ be generators of these five summands, in order of decreasing Maslov grading.

From here, we  split the analysis  into two cases: $g=2$ and $g>2.$ For both cases, we recall that there is  a differential $\partial$ on the vector space $\HFK(K)$ which lowers the Maslov (homological) grading by 1, and preserves or lowers the Alexander grading, such that the resulting Alexander-filtered chain complex $(\HFK(K),\partial)$ is quasi-isomorphic to the Alexander-filtered  complex $\CF(S^3,K)$. In particular, the homology of $(\HFK(K),\partial)$ computes the computes the Maslov-graded Heegaard Floer homology of $S^3$, i.e.
\begin{equation*}\label{eqn:homology-hfk}
H_*(\HFK(K),\partial) \cong \HF(S^3)\cong \Q_{(0)}.
\end{equation*}

First, let us consider the case that $g = 2$.
Since $\partial$ lowers Maslov grading by $1$, we may assume (up to multiplying the generators by units) that \[\partial(x_i) = x_{i+1} \textrm{ or } 0\] for each $i$. If $\partial(x_1) = 0$, then the fact that \[\dim\HF(S^3)=1\] forces  $\partial(x_2) = x_3$ and $\partial(x_4)=x_5.$ This implies that $\tau(K) = 2$.  Since $K$ is fibered of genus 2, work of Hedden  \cite{Hedden} then implies that $K$ is strongly quasipositive.  Likewise, if $\partial(x_4) = 0$, then  $\tau(K) = -2$ and  $-K$ is strongly quasipositive by a similar argument. If neither $\partial(x_1)$ nor $\partial(x_4)$ is zero, then it must be the case that  $\partial(x_1) = x_2$ and $\partial(x_4) = x_5$. We claim  this cannot happen. Let us assume it does and derive a contradiction. For this,  we  attempt to reconstruct $\CFK(K)$ from the assumption on $\partial$.  Recall that $C\{j = 0\}$ is filtered chain homotopy equivalent to $C\{i = 0\}$ \cite[Proposition 3.9]{OS:hfk}.  The horizontal and vertical components of the differential in $\CFK(K)$ are therefore as shown in Figure~\ref{fig:CFKinfty}, and there cannot be diagonal components since $\HFK(K)$ is supported in a single $\delta$-grading.
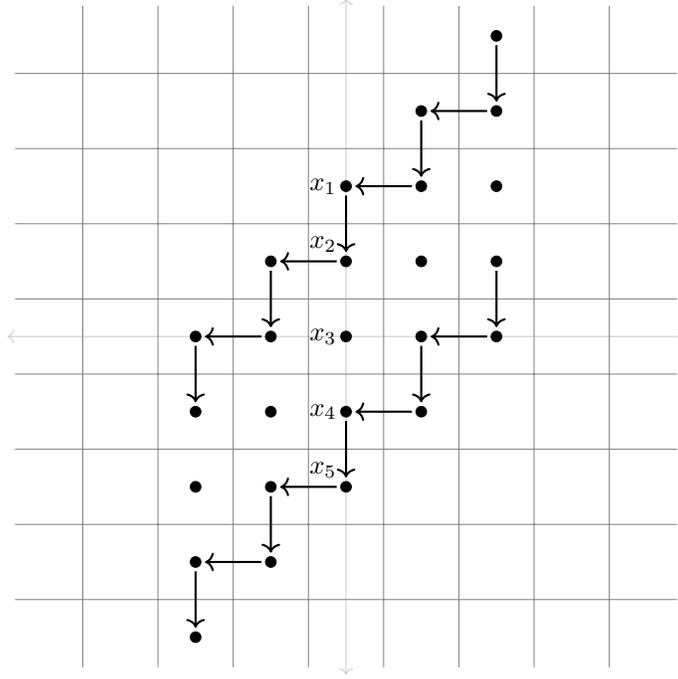
\begin{figure}
\begin{tikzpicture}

	\begin{scope}[thin, black!20!white]
		\draw [<->] (-4, 0.5) -- (5, 0.5);
		\draw [<->] (0.5, -4) -- (0.5, 5);
	\end{scope}
	\draw[step=1, black!50!white, very thin] (-3.9, -3.9) grid (4.9, 4.9);
	
	\filldraw (-1.5, .5) circle (2pt) node[] (A){};
	\filldraw (-1.5, -.5) circle (2pt) node[] (B){};
	\filldraw (-1.5, -1.5) circle (2pt) node[] (C){};
	\filldraw (-1.5, -2.5) circle (2pt) node[] (D){};	
	\filldraw (-1.5, -3.5) circle (2pt) node[] (E){};
	\filldraw (.5, 2.5) circle (2pt) node[] (F){};
	\filldraw (.5, 1.5) circle (2pt) node[] (G){};
	\filldraw (.5, 0.5) circle (2pt) node[] (H){};
	\filldraw (.5, -.5) circle (2pt) node[] (I){};	
	\filldraw (.5, -1.5) circle (2pt) node[] (J){};
	\filldraw (-0.5, 1.5) circle (2pt) node[] (K){};
	\filldraw (-0.5, 0.5) circle (2pt) node[] (L){};
	\filldraw (-0.5, -.5) circle (2pt) node[] (M){};
	\filldraw (-0.5, -1.5) circle (2pt) node[] (N){};	
	\filldraw (-0.5, -2.5) circle (2pt) node[] (O){};
	\filldraw (1.5, 3.5) circle (2pt) node[] (P){};
	\filldraw (1.5, 2.5) circle (2pt) node[] (Q){};
	\filldraw (1.5, 1.5) circle (2pt) node[] (R){};
	\filldraw (1.5, 0.5) circle (2pt) node[] (S){};	
	\filldraw (1.5, -.5) circle (2pt) node[] (T){};
	\filldraw (2.5, 4.5) circle (2pt) node[] (U){};
	\filldraw (2.5, 3.5) circle (2pt) node[] (V){};
	\filldraw (2.5, 2.5) circle (2pt) node[] (W){};
	\filldraw (2.5, 1.5) circle (2pt) node[] (X){};	
	\filldraw (2.5, .5) circle (2pt) node[] (Y){};
	\node [left] at (F) {$x_1$};
	\node [above left] at (G) {$x_2$};
	\node [left] at (H) {$x_3$};
	\node [left] at (I) {$x_4$};
	\node [above left] at (J) {$x_5$};

	\draw [thick, ->] (A) -- (B);
	\draw [thick, ->] (D) -- (E);
	\draw [thick, ->] (F) -- (G);
	\draw [thick, ->] (I) -- (J);
	\draw [thick, ->] (K) -- (L);
	\draw [thick, ->] (N) -- (O);
	\draw [thick, ->] (P) -- (Q);
	\draw [thick, ->] (S) -- (T);
	\draw [thick, ->] (U) -- (V);
	\draw [thick, ->] (X) -- (Y);
	\draw [thick, ->] (G) -- (K);
	\draw [thick, ->] (V) -- (P);
	\draw [thick, ->] (Q) -- (F);
	\draw [thick, ->] (L) -- (A);
	\draw [thick, ->] (Y) -- (S);
	\draw [thick, ->] (T) -- (I);
	\draw [thick, ->] (J) -- (N);
	\draw [thick, ->] (O) -- (D);
\end{tikzpicture}
\caption{The vertical and horizontal differentials in $\CFK(K)$ when $\dim \HFK(K,a) = 1$ for $-2 \leq a \leq 2$ and $\tau(K) \neq \pm 2$.}
\label{fig:CFKinfty}
\end{figure}
However, the differential  in this figure does not square to zero, a contradiction.

We have  shown in the case $g=2$  that, up to mirroring, $K$ is a  genus-2, fibered, strongly quasipositive knot with  $\tau(K)=2$. Let us assume for the remainder of this paragraph that $K$ (and not its mirror) is strongly quasipositive. The fact that the homology of the complex $(\HFK(K), \partial)$ is supported in Maslov grading 0
 then implies that the generator $x_1$ has Maslov grading $M=0$. The bigraded knot Floer homology of $K$ is therefore given by \begin{equation*}\label{eqn:hfk}
\HFK(K) = \Q_{(0,2)} \oplus \Q_{(-1,1)} \oplus \Q_{(-2,0)} \oplus \Q_{(-3,-1)} \oplus \Q_{(-4,-2)},
\end{equation*} which agrees with that of $T(2,5)$, as claimed.

Finally, let us consider the case that $g > 2$.
In order for the dimension of $\HF(S^3)$ to be 1, we must have that $\partial(x_1) = x_2$, $  \partial(x_3) = 0$, and  $\partial(x_4) = x_5$, by Maslov grading considerations.  The same argument as in the case $g=2$ above then shows that the differential in $\CFK(K)$ does not square to zero, a contradiction.

In summary, we have  shown that $K$ is either the figure eight or, up to mirroring, a genus-2, fibered, strongly quasipositive knot with the same  knot Floer homology as $T(2,5)$. It  remains to show that $d = 0$ in the first case and $d=2$ in the second. The first is immediate from the formula for $\Khr(4_1)$ at the beginning. We show that $d=2$ in the second case below.

First, observe that $d = s(K)/2$, where $s(K)$ is Rasmussen's invariant. To see this, a theorem of Khovanov \cite[Proposition 3.6]{KhovanovPatterns} implies that the unreduced invariant $\Kh(K)$ is thin, i.e. supported in two consecutive $\delta$ gradings, which with our normalization must be $d \pm \frac12$. Thus, the nonzero summands of $\Kh(K)$ in homological grading $0$ are in quantum gradings $2d \pm 1$, whence we deduce that $s(K)=2d$. 
In the case that $K$ is a genus-2, fibered, strongly quasipositive knot,  work of Plamenevskaya \cite[Proposition 4]{Plamenevskaya} and independently Shumakovitch \cite[Proposition 1.7]{Shumakovitch} shows that \[s(K) = 2g_4(K) = 2g(K) = 4,\] where the second equality comes from the fact that slice genus of a strongly quasipositive knot is equal to the three-genus \cite{Rudolph}. So $d=s(K)/2 = 2$ in this case, as claimed. This completes the proof of Theorem \ref{thm:main}.
\end{proof}

\bibliographystyle{alpha}
\bibliography{biblio}

\end{document}